\documentclass{amsart}
\usepackage{amsmath, amsthm, amssymb, graphicx, color, verbatim, wasysym, hyperref}
\numberwithin{equation}{section}

\usepackage[initials,nobysame]{amsrefs}

\allowdisplaybreaks[2]

\renewcommand{\S}{\mathbb{S}}
\newcommand{\Z}{\mathbb{Z}}
\newcommand{\R}{\mathbb{R}}

\newcommand{\SQ}{\mathcal{Q}}
\newcommand{\IP}[2]{\left<#1,#2\right>}
\newcommand{\vn}[1]{\lVert#1\rVert}

\newcommand{\kav}{\overline{k}}

\def\ds{\displaystyle}

\newtheorem{theorem}{Theorem}[section]
\newtheorem*{theorem*}{Theorem}
\newtheorem{prop}[theorem]{Proposition}
\newtheorem{lem}[theorem]{Lemma}
\newtheorem{cor}[theorem]{Corollary}
\theoremstyle{remark}
\newtheorem{rmk}[theorem]{Remark}

\begin{document}

\title{The free elastic flow for closed planar curves}

\author[T.~Miura]{Tatsuya Miura}
\address[T.~Miura]{Department of Mathematics, Graduate School of Science, Kyoto University, Kitashirakawa Oiwake-cho, Sakyo-ku, Kyoto 606-8502, Japan}
\email{tatsuya.miura@math.kyoto-u.ac.jp}

\author[G.~Wheeler]{Glen Wheeler}
\address[G.~Wheeler]{Institute for Mathematics and its Applications \\
University of Wollongong\\
Northfields Avenue\\
Wollongong, NSW, 2522, Australia}
\email{glenw@uow.edu.au}


\begin{abstract}
The free elastic flow is the $L^2$-gradient flow for Euler's elastic energy, or equivalently the Willmore flow with translation invariant initial data.
In contrast to elastic flows under length penalisation or preservation, it is more challenging to study the free elastic flow's asymptotic behaviour, and convergence for closed curves is lost.
In this paper, we nevertheless determine the asymptotic shape of the flow for initial curves that are geometrically close to circles, possibly multiply-covered, proving that an appropriate rescaling smoothly converges to a unique round circle.
\end{abstract}

\maketitle

\section{Introduction}

Let $\gamma:\S^1\rightarrow\R^2$ be a smooth immersion of $\S^1:=\R/\Z$ into the plane.
Set
\[
	E[\gamma] = \frac12\int k^2\,ds
\]
to be Euler's elastic energy.
The \emph{free elastic flow} is the steepest descent $L^2(ds)$-gradient flow for Euler's elastic energy given by the evolution equation
\begin{equation}
\label{FEF}
\tag{FEF}
\partial_t\gamma = -\Big(k_{ss} + \frac12k^3\Big)\nu\,,\qquad \gamma(\cdot,0) = \gamma_0,
\end{equation}
where $\gamma:\S^1\times[0,T)\rightarrow\R^2$ is a one-parameter family of smooth
immersed curves.
Here we use the convention that $\nu$ is the
inward-pointing unit normal and $k$ is the curvature scalar (see also Section \ref{subsec:notation}).

There is a straightforward connection between the Willmore flow for surfaces and the free elastic flow.
Let $f_0:\Sigma\rightarrow\R^3$ be a smooth immersion of a closed surface $\Sigma$.
The \emph{Willmore flow} is the steepest descent $L^2(d\mu)$-gradient flow of the Willmore functional
\[
W[f] = \frac12 \int H^2\,d\mu\,,
\]
with velocity
\begin{equation}
\label{WF}
\tag{WF}
\partial_tf = -\bigg(\Delta H + \frac12H|A^\circ|^2\bigg)\nu\,,
\end{equation}
where $f:\Sigma\times[0,T)\rightarrow\R^3$ is a one-parameter family of immersions.
Here we use $\Delta$ for the Laplace--Beltrami operator on $f(\cdot,t)$ with respect to the induced metric, $H$ for the mean curvature, $A^\circ$ for the trace-free second fundamental form, and $\nu$ again for the normal vector.
We refer the interested reader to the seminal work of Kuwert--Sch\"atzle \cite{KS1,KS2,KS3} for fundamental results on the Willmore flow of closed surfaces (see also \cite{DMSS23}).

The free elastic flow is related to the Willmore flow not of closed surfaces but of complete surfaces.
The graphical sub-case of this has been considered by Koch--Lamm \cite{KL12}.
Among other results, they show that the Willmore flow of entire graphs with small Lipschitz norm in any dimension exist globally in time.
In a very recent contribution, Li \cite{L24} shows that  the Willmore flow with small initial energy and sub-Euclidean volume growth subconverges to a plane.
Let us relate these results to our setting here.
If we assume that $\Sigma$ is a cylinder, and $f_0$ is translation-invariant, then the flow reduces to the flow of its profile curve $\gamma$, and the Willmore flow \eqref{WF} for $f$ becomes the free elastic flow \eqref{FEF} for $\gamma$.
We note that the Willmore energy of the initial surface $f_0$ is infinite, and that our flow can not be written as a graph (with small Lipschitz norm or otherwise).

Standard theory for parabolic equations gives that from any smooth initial curve the free elastic flow above exists uniquely for a short time,
and is a smooth one-parameter family of immersions.
In addition, in their celebrated study \cite{DKS02}, Dziuk--Kuwert--Sch\"atzle proved that the flow always extends globally in time.

\begin{theorem}[{\cite[Theorem 3.2]{DKS02}}]
\label{thm:global_existence}
Let $\gamma_0:\S^1\rightarrow\R^2$ be a smooth immersed curve.
Then the free elastic flow with initial data $\gamma_0$ exists uniquely for all time $t\in[0,\infty)$.
\end{theorem}

In fact, Dziuk--Kuwert--Sch\"atzle proved global existence not only for the free elastic flow but also for the length-penalised and the length-preserving elastic flow for closed curves.
They further proved that the latter two flows always converge to stationary solutions (elasticae) in a certain sense.
Since then convergence results are very well-studied for these constrained or penalised flows \cite{MP21}, including in similar situations \cite{DLLPS18,DLP17,DLP19,DPS16,DS17,L12,LLS15,MS22,NO17,Polden1996,Pozzetta22}.

In contrast to the wealth of results under length penalisation or preservation, nothing has been known about the asymptotic shape of the free elastic flow for closed planar curves, even though the free elastic flow would be the most natural $L^2$-gradient flow for Euler's elastic energy.
To gain some understanding, it is helpful to think of some special solutions (see Section \ref{subsec:special} for details).

\begin{enumerate}
    \item (Stationary solutions) There are \emph{no} stationary solutions of closed curves.
    \item (Expanding solutions)
    Circles expand with radius $\rho(t) = (\rho(0)^4+2t)^{\frac14}$.
    Moreover, the lemniscate of Bernoulli also expands self-similarly \cite{EGMWW15}, again with length $\sim t^{\frac14}$.
\end{enumerate}

Even among open curves, only stationary solutions are added to the list: straight lines and rectangular elasticae. It is not known whether or not there are further non-trivial special solutions, in particular translators.
In the closed case, the known solutions already beg the stability question: If a free elastic flow is close to a multiply-covered circle in a scale-invariant sense, is it asymptotic to a multiply-covered circle?
Our main result answers this question in the positive.
We emphasise that the stability of multiply-covered circles is in general a delicate issue.

We also mention that the free elastic flow is previously studied in \cite{WW17} for open curves subject to a free boundary condition on parallel lines.
This class admits stationary solutions such as straight lines, thus being significantly different from our setting.
In \cite[Theorem 1.4]{WW17} stability of straight segments is obtained.

\subsection{Main result}

Let $L$ denote the length functional, and set $L(t):=L[\gamma(\cdot,t)]$.

\begin{theorem}[Geometric stability of $\omega$-circles]
\label{thm:stability}
    There exists a positive constant $\varepsilon=\varepsilon(\omega)>0$ with the following property:
    Let $\gamma_0:\S^1\rightarrow\R^2$ be a smooth immersed curve with (absolute) turning number $\omega=|\frac{1}{2\pi}\int_{\gamma_0}k\,ds|$ such that
    \begin{equation}
    \label{EQgeomsmallintro}
    L[\gamma_0]^3\int_{\gamma_0} k_s^2\,ds < 
    \varepsilon\,.
    \end{equation}
    Then the free elastic flow $\gamma:\S^1\times[0,\infty)\rightarrow\R^2$ with initial data $\gamma_0$ is asymptotically an $\omega$-circle, in the sense that the rescaled flow $\eta(\cdot,t) := \frac{1}{L(t)}\gamma(\cdot,t)$ smoothly converges, up to reparametrisation, as $t\to\infty$ to an $\omega$-fold circle with radius $\frac{1}{2\omega\pi}$ centred at the origin.
\end{theorem}

\begin{rmk}
The convergence rate of the solution is polynomial in time, with specific exponent governed by a universal constant $c_1>0$, which is in particular independent of $\omega$. For example, we have
\[
    ||k_{s^m}||_\infty
    \le C(1+t)^{-(m+1+c_1)/4}
\,,
\]
where $C$ depends on $m$, $\omega$, $\gamma_0$ (see Lemma \ref{LMkssest} for details).
\end{rmk}

Special solutions indicate why a stability result may be difficult to prove.
The key point is that length grows at a specific rate when close to a circle.
In order to control the flow using the condition \eqref{EQgeomsmallintro}, there are two key tasks.
First, we need to prove a sharp length estimate; that is, one with the exact rate of an $\omega$-circle.
This is achieved by studying the evolution of length carefully, applying a decomposition of the curvature into its average (which depends on the length and turning number only) and average-free parts.
Second, we need to show that the condition \eqref{EQgeomsmallintro} improves under the flow.
This second part requires some inspirational choice of functional adapted to the problem.
Here, we take $\int k_s^2\,ds$ and normalise the scale by using the elastic energy; that is, we consider
\[
\SQ(t) := \frac{\ds\int k_s^2\,ds}{\ds\bigg( \int k^2\,ds \bigg)^3}
\,.
\]
This idea departs from the standard way of working with normalisation by length (see \cite{DKS02} and also \cite{AMWW,W13,MO21} for other flows).
Once $\SQ(t)$ is shown to decay, then we also obtain decay of $\varepsilon(t) = L^3\int k_s^2\, ds$.

Another remarkable point is that the translation of the flow is also controlled.
Not only does the rescaled flow not wander off to infinity, but also the centre of the final circle must be the origin.
In terms of the original flow, the centre of mass must be confined to a region of radius $o(t^{\frac14})$.

\begin{rmk}
There are a variety of approaches to proving convergence of flows under given initial conditions.
One prominent approach is via center manifold analysis and linearisation, as famously demonstrated, for example, by \cite{EMS98} for the surface diffusion flow and \cite{S01} for the Willmore flow.
This approach does not yield an explicit constant but does have the advantage of enabling much weaker regularity requirements --- see for example \cite{EM10} for the surface diffusion flow.
One key difficulty here in making this approach work is that, unlike surface diffusion or Willmore flows for closed surfaces, the flow in question never converges. 
Rescaling (or alternative reasoning) would be required to address this issue.
Moreover, it is unclear whether decay estimates in terms of the original time parameter can be derived purely through a linearisation approach, even with an appropriate rescaling. 
Instead, estimates may need to follow the lines of Proposition \ref{PRestf}, Lemma \ref{LMkssest}, and the proof of Theorem \ref{thm:stability}.
\end{rmk}

\subsection{Open problems}

As our study provides the very first asymptotic analysis of the free elastic flow for closed planar curves, there still remain many interesting open problems.

While we deal with all non-zero turning numbers, the case of $\omega=0$ remains open.
Here, the model solution is not an $\omega$-circle but rather the lemniscate of Bernoulli.
We conjecture that, as with the circle case considered here, an analogous stability result holds for the lemniscate.

Our theorem yields in an obvious way the non-existence of non-circular expanding solutions satisfying \eqref{EQgeomsmallintro}.
In order to better understand the behaviour of the free elastic flow, it is absolutely crucial to classify all self-similar solutions, regardless of curvature condition.

In our proof, a crucial role is played by sharp control of the length.
A known general estimate of Dziuk--Kuwert--Sch\"atzle \cite[(3.11)]{DKS02} shows that the length grows at most linearly in time.
We do not expect the linear upper bound is optimal; in fact, the derivation of \cite[(3.11)]{DKS02} even implies that for a universal $C>0$, 
\begin{equation}\label{eq:length_sublinear}
    L(t) \leq L(0)+ C \int_0^t E(\hat{t})^3\,d\hat{t}\,,
\end{equation}
and the energy $E$ decreases in general.
All the known explicit solutions, the circles and the lemniscate of Bernoulli, have computable length growth $\sim t^{\frac14}$.
Our proof implies that all free elastic flows satisfying \eqref{EQgeomsmallintro} have this same length growth (see Proposition \ref{prop:length_sharp_estimate}).
This leads one to the question: does there exist any solution to the free elastic flow with length growth different to $O(t^{\frac14})$?

Finally, we mention Huisken's problem (see \cite{MPP21}).
This asks after the existence of an elastic flow that begins contained in the upper half plane but `migrates' to the lower half plane at a positive time.
Very recently \cite{KM24} Huisken's problem has been resolved for certain length-preserving flows of open curves, but the original problem for length-penalised flows of closed curves is left open.
It also remains open for the free elastic flow.
One might suspect that migration can not occur for the free elastic flow, since the flow tends to inflate all curves.
However there is the possibility that the centre of mass along a solution is not well-controlled and in fact diverges to infinity.
Indeed, controlling the centre of mass of a flow is a delicate point, and in higher-order flows typically is possible only after very strong convergence allows an integration of the evolution equation in time.
Thus the two problems are connected: (1) Does there exist a free elastic flow with centre of mass diverging to infinity? (2) Does there exist a migrating free elastic flow?

\section*{Acknowledgements}

The first author is supported by JSPS KAKENHI Grant
Numbers JP21H00990, JP23H00085, and JP24K00532.

\section{Preliminaries}
\label{sec:prelim}

\subsection{Notation}\label{subsec:notation}

For a curve $\gamma:\S^1\to\R^2$ ($u\mapsto\gamma(u)$), let $\partial_s:=|\partial_u\gamma|^{-1}\partial_u$ be the arc-length derivative, $\tau:=\partial_s\gamma$ the unit tangent vector, $\nu:=\text{rot}\,\tau$ the inward-pointing unit normal vector (the counterclockwise rotation of $\tau$ through angle $\pi/2$), $\kappa:=\partial_s^2\gamma$ the curvature vector, and $k:=\IP{\kappa}{\nu}$ the signed curvature, where $\IP{\cdot}{\cdot}$ stands for the Euclidean inner product.
We also frequently use subindex-type notation like $\gamma_u$, $\gamma_{ss}$ to denote derivatives.
For abbreviation we will also write
\[
F := k_{ss} + \frac12k^3\,.
\]
In particular, the free elastic flow is expressed as $\partial_t\gamma=-F\nu$.
Let us also introduce the norms $\vn{f}_p:=(\int_{\S^1} |f|^p ds)^{1/p}$, where $ds:=|\gamma_u|du$, and $\vn{f}_\infty:=\text{ess}\sup_{u\in\S^1}|f(u)|$.
The average is denoted by 
\[\overline{f}:=\frac{1}{L[\gamma]}\int_{\S^1} f\,ds\,.
\]
We will often drop the domain integral if it is not confusing.

For a family of curves $\gamma:\S^1\times[0,T)\to\R^2$ ($(u,t)\mapsto\gamma(u,t)$) we also use the same notation $\partial_s$ for spatial derivatives of $\gamma(\cdot,t)$, and $\partial_t$ for time derivatives.
The spatial integral $\int f\, ds$ is understood at each time slice $t$.

\subsection{Special solutions}
\label{subsec:special}

\subsubsection{Nonexistence of stationary solutions}

Every critical point of $E$ must be graphical \cite{Miura21} so it can not be closed.
From a broader point of view, according to the standard classification of Euler's elasticae (that is, critical points of $E+\lambda L$) \cite{Singer08,MY24}, the only closed candidates are the circle and the figure-eight elastica.
Both of these have non-zero Lagrange multiplier $\lambda\neq0$, and so are not free elasticae (with $\lambda=0$).
This observation indicates that the known convergence results strongly rely on the penalisation or preservation of length.

\subsubsection{Circles}

Let $N$ be a positive integer, and consider the family of evolving $N$-circles defined by
$$\gamma^N(u,t)=\rho(t)(\cos{2\pi Nu},\sin{2\pi Nu}),$$ 
which has turning number $\omega=N$.
Since $\IP{\partial_t\gamma^N}{\nu}=-\rho'$, $k=1/\rho$, and $k_{ss}\equiv0$, this family satisfies the free elastic flow if and only if
\[
\rho'(t)=\frac{1}{2\rho(t)^3}.
\]
This ODE for the radius function $\rho$ has the unique solution
\[
\rho(t)=(\rho(0)^4+2t)^\frac{1}{4}.
\]

\subsubsection{Lemniscate of Bernoulli.}

The lemniscate of Bernoulli
$$\beta(u)=\frac{1}{1+\sin^2{u}}\left( \cos{u},\frac{1}{2}\sin{2u} \right)$$
has turning number zero, and as computed in \cite{EGMWW15}, satisfies
\[
	k(u)^3
	= -27\IP{\beta(u)}{\nu(u)}
 \quad \text{and} \quad
	k_{ss}(u)
	= 6\IP{\beta(u)}{\nu(u)}\,.
\]
Consider the family
\[
\gamma(u,t) := h(t)\beta(u).
\]
Then
\begin{align*}
	\IP{\gamma_t + \Big(k_{ss}+\frac12 k^3\Big)\nu}{\nu}
	&= h'(t)\IP{\beta}{\nu} + h^{-3}(t) \Big(6\IP{\beta}{\nu} - \frac{27}{2}\IP{\beta}{\nu}\Big)
	\\
	&= h^{-3}(t)\IP{\beta}{\nu}\Big(
		h^3(t)h'(t) - \frac{15}{2}
		\Big)
	\\
	&= \frac{h^{-3}(t)}{4}\IP{\beta}{\nu}\Big(
		(h^4(t))' - 30
		\Big).
\end{align*}
Thus, the family is a self-similar free elastic flow if
\[
	h(t) = (h(0)^4+30t)^\frac14
	\,.
\]
In particular, the map
$
	(u, t) \mapsto
	(1+30t)^\frac14\frac{\cos{u}}{1+\sin^2{u}} (1, \sin{u} )
$
is a self-similar expanding solution to the free elastic flow, with initial data the lemniscate of Bernoulli.

\section{Geometric stability of circles}

The main goal of this section is to prove Theorem \ref{thm:stability}.
We now introduce the scale invariant quantity
\[
\varepsilon(t) := L^3||k_s||_2^2\,.
\]
The key idea now is to use $\varepsilon(t)$ to control the asymptotic shape of the flow.
We will first prove the smallness-preservation and decay of $\varepsilon(t)$, which already proves that the flow tends to be circular, and then obtain the control of the centre.

We start by computing the evolution of $\vn{k_s}_2^2$ in time.

\begin{lem}
\label{LMksevol}
Let $\gamma:\S^1\times[0,T)\rightarrow\R^2$ be a free elastic flow.
Then
\[
	\frac{d}{dt} \int k_s^2\,ds =  
    - 2\int k_{sss}^2\,ds
    + 5\int k_{ss}^2k^2\,ds
    - \frac{5}{3}\int k_s^4\,ds
    - \frac{11}2\int k_s^2k^4\,ds
\,.
\]
\end{lem}
\begin{proof}
The commutator for arc-length and time is given by
\[
    [\partial_t,\partial_s]
    = -Fk\,\partial_s
    \,,
\]
where we recall $F = k_{ss} + \frac12k^3$.
To illustrate how this is used, consider the evolution of the tangent vector:
\begin{align*}
\tau_t
&= \partial_t\partial_s \gamma \\
&= [\partial_t,\partial_s] \gamma + \partial_s\partial_t\gamma \\
&= -Fk\tau + (-F\nu)_s \\
&= -F_s\nu
 \,.
\end{align*}
Similarly we find the evolution of the unit normal vector, the curvature vector, curvature scalar, and the derivative of the curvature scalar:
\begin{align*}
\nu_t
&= F_s\tau,
\\
\kappa_t
&= [\partial_t,\partial_s] \tau + \partial_s\partial_t \tau
\\
&= -Fk\kappa + (-F_s\nu)_s
\\\
&= -(F_{ss} + Fk^2)\nu + F_sk\tau,
\\
k_t &= \IP{\kappa}{\nu}_t
\\
&= -(F_{ss} + Fk^2),
\\
k_{st} &= [\partial_t,\partial_s] k + k_{ts}
\\
&= -Fk_sk - (F_{ss} + Fk^2)_s
\\
&= -(F_{sss} + F_sk^2 + 3Fk_sk)\,.
\end{align*}
In addition, the evolution law of the arc-length measure $ds$ is obtained as follows:
\[
ds_t
 = \frac{\IP{\gamma_{ut}}{\gamma_{u}}}{|\gamma_u|}du
 = \frac{\IP{(-F\nu)_u}{\gamma_{u}}}{|\gamma_u|}du
 = -F\frac{|\gamma_u|\IP{-k\tau}{\gamma_{u}}}{|\gamma_u|}du
 = Fk\,ds\,,
\]

Now we combine the above to find
\begin{align*}
\frac{d}{dt} \int k_s^2\,ds
&=
    - 2\int (F_{sss} + F_sk^2 + 3Fk_sk)k_s
        \,ds
    + \int Fk_s^2k\,ds
\\&=
    \int (-2F_{sss} - 2F_sk^2 - 5Fk_sk)k_s
        \,ds
        \,.
\end{align*}
Using $F = k_{ss} + \frac12k^3$ and integrating by parts yield
\begin{align*}
\frac{d}{dt} \int k_s^2\,ds
&=
    - 2\int k_{sss}^2\,ds
    + \int (k^3)_{ss}k_{ss}
        \,ds
    - 2\int k_{sss}k_sk^2
        \,ds
    - \int (k^3)_s k_sk^2
        \,ds
\\&\qquad
    - 5\int k_{ss}k_s^2k
        \,ds
    - \frac52\int k_s^2k^4
        \,ds.
\end{align*}
Using $(k^3)_s=3k^2k_s$ and $(k^3)_{ss}=3k^2k_{ss}+6k_s^2k$ for the fourth and second terms, respectively, and integrating by parts for the third term further imply
\begin{align*}
\frac{d}{dt} \int k_s^2\,ds
&=
    - 2\int k_{sss}^2\,ds
    + 3\int (k_{ss}k^2 + 2k_s^2k)k_{ss}
        \,ds
    + 2\int k_{ss}(k_{ss}k^2 + 2k_s^2k)
        \,ds
\\&\qquad
    - 5\int k_{ss}k_s^2k
        \,ds
    - \frac{11}2\int k_s^2k^4
        \,ds
\\
&=
    - 2\int k_{sss}^2\,ds
    + 5\int k_{ss}^2k^2
        \,ds
    + 5\int k_{ss}k_s^2k
        \,ds
    - \frac{11}2\int k_s^2k^4
        \,ds
\,.
\end{align*}
Integrating by parts for the third term via $k_{ss}k_s^2=(\frac{1}{3}k_s^3)_s$ completes the proof.
\end{proof}

From now on we will often use the following $L^2$ and $L^\infty$ Poincar\'e--Sobolev--Wirtinger inequalities as in \cite{W13}:

\begin{lem}
    If $f:\S^1\to\R$ has zero average $\overline{f}=0$, then 
    \begin{equation}\label{eq:Poincare}
        \text{$\|f\|_2^2\leq \frac{L^2}{4\pi^2}\|f_s\|_2^2$ \quad and \quad $\|f\|_\infty^2\leq\frac{L}{2\pi}\|f_s\|_2^2$.}
    \end{equation}
\end{lem}

This will typically be applied to the derivatives of $k$, or the average-free part $k-\bar{k}$.

To this end we apply the average-free/average decomposition to Lemma \ref{LMksevol}.

\begin{lem}
\label{LMksevol2}
Let $\gamma:\S^1\times[0,\infty)\rightarrow\R^2$ be a free elastic flow.
Then
\begin{align*}
\frac{d}{dt} \int k_s^2\,ds
	&\leq
	 -\frac{1}{8}\int k_{sss}^2\,ds - \frac{13}{6}\kav^4 \int k_s^2\,ds\\
  &\qquad - 22\kav^3\int (k-\kav)k_s^2\,ds + 5\int (k-\kav)^2k_{ss}^2\,ds
	 + 10\kav\int (k-\kav)k_{ss}^2\,ds\,.
\end{align*}
\end{lem}
\begin{proof}
Note first the following formulae that follow in a straightforward way by using $k = (k-\kav) + \kav$:
\[
	5\int k^2k_{ss}^2\,ds
	 = 5\int (k-\kav)^2k_{ss}^2\,ds
	 + 10\kav\int (k-\kav)k_{ss}^2\,ds
	 + 5\kav^2\int k_{ss}^2\,ds\,,
\]
\begin{align*}
    	-\frac{11}{2}\int k^4k_s^2\,ds
	&= -\frac{11}{2}\int (k-\kav)^4k_s^2\,ds - 22\kav\int (k-\kav)^3k_s^2\,ds
		 \\
        &\qquad - 33\kav^2\int (k-\kav)^2k_s^2\,ds - 22\kav^3\int (k-\kav)k_s^2\,ds
		 - \frac{11}{2}\kav^4\int k_s^2\,ds\,.
\end{align*}
In addition, for $a_0,a_1>0$,
\[
	 5\kav^2\int k_{ss}^2\,ds = -5\kav^2\int k_{sss}k_s\,ds
	\le a_0 \int k_{sss}^2\,ds + \frac{25}{4a_0}\kav^4 \int k_s^2\,ds
 \,,
\]
\[
	-22\kav\int (k-\kav)^3k_s^2\,ds
	\le 22a_1\kav^2\int (k-\kav)^2k_s^2\,ds + \frac{11}{2a_1}\int (k-\kav)^4k_s^2\,ds\,.
\]
Now we apply the first two identities to the corresponding terms in Lemma \ref{LMksevol}, and further apply the last two estimates to deduce that
\begin{align*}
\frac{d}{dt} \int k_s^2\,ds
	&\le
	 -(2-a_0)\int k_{sss}^2\,ds
	 -\frac53\int k_s^4\,ds
	 -\frac{11}2\Big(1 - \frac{1}{a_1}\Big)\int (k-\kav)^4k_s^2\,ds
\\&\qquad
	 -(33 - 22a_1)\kav^2\int (k-\kav)^2k_s^2\,ds
	+ \Big(\frac{25}{4a_0}- \frac{11}2\Big)\kav^4 \int k_s^2\,ds
\\&\qquad
	 - 22\kav^3\int (k-\kav)k_s^2\,ds + 5\int (k-\kav)^2k_{ss}^2\,ds
	 + 10\kav\int (k-\kav)k_{ss}^2\,ds\,.
\end{align*}
Taking $a_0=\frac{15}{8}$ and $a_1=1$ and deleting some non-positive terms complete the proof.
\end{proof}

\begin{rmk}
    The above $a_1$ is chosen just to reduce the number of terms; in fact, there is essentially no need to take $a_1$ because the term $\kav\int (k-\kav)^3k_s^2\,ds$ can be absorbed into the first term under the smallness assumption imposed later.
    On the other hand, the choice of $a_0$ is much more delicate, and even the numerical value is important.
    Our subsequent arguments require not only $a_0<2$ for making the first term negative, but also $a_0>\frac{25}{16}$ to have the coefficient of $\kav^4\int k_s^2\,ds$ less than $-\frac{3}{2}$ (this will be used in the proof of Proposition \ref{PRPkeyestimate}).
\end{rmk}

Now we estimate $\varepsilon(t)$.
The first task is to show that if it starts small enough, it remains small for an estimable period of time.
An important step toward this is the following estimate.

\begin{lem}\label{LMsmallnessks}
Let $\gamma:\S^1\times[0,\infty)\rightarrow\R^2$ be a free elastic flow.
Then
\begin{align*}
	\frac{d}{dt} \int k_s^2\,ds
	\le -\left(\frac{1}{8}-\frac{5}{8\pi^3}\varepsilon(t)-\frac{11\omega^3+5\omega}{\sqrt{2\pi^3}}\sqrt{\varepsilon(t)}\right)\int k_{sss}^2\,ds - \frac{13}{6}\kav^4\int k_s^2\,ds\,.
\end{align*}
In particular, at any $t$ such that
\[
\varepsilon(t)
\leq \varepsilon_*(\omega) := 
\frac{8\pi^3}{25}\left(\sqrt{\big( 11\omega^3+5\omega \big)^2+\frac{5}{16}}- \big( 11\omega^3+5\omega \big)\right)^2
\]
we have
\[
	\frac{d}{dt} \int k_s^2\,ds
	\leq -\frac{1}{16}\int k_{sss}^2\,ds - \frac{13}{6}\kav^4\int k_s^2\,ds\,.
\]
\end{lem}

\begin{proof}
We deduce from \eqref{eq:Poincare} that
\[\vn{k-\kav}_\infty^2 \leq \frac{L}{2\pi}\|k_s\|_2^2, \quad \|k_s\|_2^2 \leq\frac{L^2}{4\pi^2}\|k_{ss}\|_2^2, \quad \|k_{ss}\|_2^2 \leq \frac{L^2}{4\pi^2}\|k_{sss}\|_2^2.
\]
These can be used to estimate the last three terms in Lemma \ref{LMksevol2}.
For example, for the first term,
\begin{align*}
    - 22\kav^3\int (k-\kav)k_s^2\,ds &\leq 22\left(\frac{2\omega\pi}{L}\right)^3\vn{k-\kav}_\infty\|k_s\|_2^2 \\
    & \leq 22\left(\frac{2\omega\pi}{L}\right)^3\sqrt{\frac{L}{2\pi}}\vn{k_s}_2\left(\frac{L^2}{4\pi^2}\right)^2\vn{k_{sss}}_2^2\\
    & = \frac{11\omega^3}{\sqrt{2\pi^3}}\sqrt{\varepsilon(t)}\|k_{sss}\|_2^2\,.
\end{align*}
Similarly, for the second and third terms,
\[
	 5\int (k-\kav)^2k_{ss}^2\,ds
	\le 5\vn{k-\kav}_\infty^2\|k_{ss}\|_2^2
	\leq \frac{5}{8\pi^3}\varepsilon(t)\vn{k_{sss}}_2^2\,,
\]
\[
	 10\kav\int (k-\kav)k_{ss}^2\,ds
	\le \frac{20\omega\pi}{L} \vn{k-\kav}_\infty \vn{k_{ss}}_2^2
	\le  \frac{5\omega}{\sqrt{2\pi^3}} \sqrt{\varepsilon(t)} \vn{k_{sss}}_2^2
\,.
\]
Applying the estimates to the result of Lemma \ref{LMksevol2} yields
\begin{align*}
	\frac{d}{dt} \int k_s^2\,ds
	\le -\left(\frac{1}{8}-\frac{5}{8\pi^3}\varepsilon(t)-\frac{11\omega^3+5\omega}{\sqrt{2\pi^3}}\sqrt{\varepsilon(t)}\right)\vn{k_{sss}}_2^2.
\end{align*}
Computing the positive root $x_0$ of $\frac{1}{16}-\frac{5}{8\pi^3}x^2-\frac{11\omega^3+5\omega}{\sqrt{2\pi^3}}x$, we find that if
$$\sqrt{\varepsilon(t)}\leq x_0=\frac{\sqrt{8\pi^3}}{5}\left(\sqrt{\big( 11\omega^3+5\omega \big)^2+\frac{5}{16}}- \big( 11\omega^3+5\omega \big)\right)\,,$$
then the prefactor of $\vn{k_{sss}}_2^2$ is bounded above by $-\frac{1}{16}$.
\end{proof}

Now we wish to study the flow under the specific geometric closeness condition
\begin{equation}
\label{EQgeomsmall}
\varepsilon(t) \le \varepsilon_*(\omega)
\,.
\end{equation}

The first immediate impact of working in the smallness regime is that we can obtain a sharp length estimate.
Recall that the known length bound without any smallness assumption is of the form
\[
\frac{2\pi^2}{E[\gamma_0]} \leq L(t) \leq L(0)+C\big(E[\gamma_0]\big)\,t.
\]
The lower bound simply follows by the decreasing property of $E(t):=E[\gamma(\cdot,t)]$, and the estimate $L(t)E(t)\geq 2\pi^2$ obtained by the Cauchy--Schwarz inequality and Fenchel's theorem.
The upper bound is obtained by Dziuk--Kuwert--Sch\"atzle \cite[(3.11)]{DKS02}.
The above length estimate is quite rough at least for our purpose, as all the known special solutions expand at rate $t^{\frac{1}{4}}$.
In the smallness regime, we can drastically improve the estimate.

\begin{prop}\label{prop:length_sharp_estimate}
Let $\sigma\geq0$.
Let $\gamma:\S^1\times[0,\infty)\rightarrow\R^2$ be a free elastic flow with $\varepsilon(0)\leq\sigma$.
Set $T_\sigma:=\inf\{ t\geq0 \mid \varepsilon(t)>\sigma\}\in[0,\infty]$.
Then, for all $t\in[0,T_\sigma)$,
\[
    |L(t)^4 - L(0)^4 - 32\omega^4\pi^4\,t | \leq \sigma \hat{C}(\sigma,\omega)\, t,
\]
where $\hat C(\sigma,\omega)$
is a quadratic polynomial in $\sqrt{\sigma}$ with positive coefficients. 
\end{prop}

\begin{proof}
Using $k=(k-\kav)+\kav$, we calculate
\begin{align*}
    & \frac{d}{dt}(L(t)^4)\\
    &= 4L^3\int\, ds_t\\
    & = -4L^3\int k_s^2\,ds
        + 2L^3\int k^4\,ds \\
    &= -4\varepsilon(t) + 2L^3\int \Big( (k-\kav)^4 + 4(k-\kav)^3\kav + 6(k-\kav)^2\kav^2 + 4(k-\kav)\kav^3 + \kav^4 \Big)\,ds
\notag
\\  &= -4\varepsilon(t) + 2L^3\int \Big( (k-\kav)^4 + 4(k-\kav)^3\kav + 6(k-\kav)^2\kav^2\Big)\,ds
        + 32\omega^4\pi^4
\,.
\label{EQlengthevo}
\end{align*}
Now we use \eqref{eq:Poincare} to estimate the integral term:
\begin{align*}
    \Big| 2L^3\int (k-\kav)^4 &+ 4(k-\kav)^3\kav + 6(k-\kav)^2\kav^2\,ds \Big|\\
    &\le \Big(2L^3||k-\kav||_\infty^2 + 8L^3\kav||k-\kav||_\infty + 12L^3\kav^2\Big)\int (k-\kav)^2\,ds
   \notag \\
    &\le \Big(\frac{L^4}{\pi}||k_s||_2^2
        + 16L^2\omega\pi\sqrt{\frac{L}{2\pi}}||k_s||_2
        + 48\omega^2\pi^2L\Big)\frac{L^2}{4\pi^2}\int k_s^2\,ds
    \notag \\
    &\le \Big(\frac{1}{4\pi^3}\varepsilon(t)
        + \sqrt{\frac{8\omega^2}{\pi^3}}\sqrt{\varepsilon(t)}
        + 12\omega^2\Big) \varepsilon(t)\,.
\end{align*}
Letting
$C(\sigma,\omega):=\frac{1}{4\pi^3}\sigma
+ \sqrt{\frac{8\omega^2}{\pi^3}}\sqrt{\sigma}
+ 12\omega^2$,
using $\varepsilon(t)\leq\sigma$ and integrating give
\[
-\sigma\big(4+C(\sigma,\omega) \big)t \leq L(t)^4-L(0)^4 -32\omega^4\pi^4\,t \leq  \sigma C(\sigma,\omega)t.
\]
Letting $\hat{C}:=C+4$ finishes the proof.
\end{proof}

\begin{rmk} 
    For an $\omega$-circle, equality is achieved for all time for $\sigma=0$ and $T_\sigma=\infty$.
    Not only that, it turns out that any flow with bounded $\varepsilon(t)$ must have length growth $\lesssim t^\frac{1}{4}$, and in addition if $\varepsilon(t)$ remains small the length growth is exactly $\sim t^\frac{1}{4}$.
\end{rmk}

We use the sharp length upper bound to obtain a life-span estimate for ensuring the smallness of $\varepsilon(t)$.
For positive constants $0<\varepsilon<\sigma$, we define $\delta_*(\varepsilon,\sigma)>0$ to be the unique positive root of
\begin{equation}
\label{EQrootp}\nonumber
P(t) := \varepsilon
\Big(
1
+ 
  \big(
  \sigma\hat{C}(\sigma, \omega) 
  + 32\omega^4\pi^4\big)t
\Big)^{\frac34}
- \sigma\,,
\end{equation}
that is,
\begin{equation}\label{eq:def_delta}
    \delta_*(\varepsilon,\sigma)=\delta_*(\varepsilon,\sigma;\omega) := \frac{1}{\sigma\hat{C}(\sigma, \omega) 
  + 32\omega^4\pi^4} \left( \left(\frac{\sigma}{\varepsilon}\right)^\frac{4}{3}-1 \right)\,.
\end{equation}
How this quantity emerges can be seen immediately in the following proof.

\begin{lem}\label{lem:delta_*}
    Let $\sigma\in(0,\varepsilon_*(\omega)]$.
    Let $\gamma:\S^1\times[0,\infty)\rightarrow\R^2$ be a free elastic flow satisfying $\varepsilon(0)<\sigma$.
    Then, for all $t\in[0,L(0)^4\delta_*(\varepsilon(0),\sigma)]$,
    \[
    \varepsilon(t)\leq\sigma \leq \varepsilon_*(\omega).
    \]
\end{lem}

\begin{proof}
    Let $T_\sigma:=\inf\{t\geq0 \mid \varepsilon(t)>\sigma\}\in(0,\infty]$.
    It suffices to prove that $T_\sigma\geq L(0)^4\delta_*(\varepsilon(0),\sigma)$.
    Suppose on the contrary that $T_\sigma< L(0)^4\delta_*(\varepsilon(0),\sigma)$.
    By $T_\sigma<\infty$ (and continuity), for all $t\in[0,T_\sigma]$, we have
    $
    \varepsilon(t)\leq \sigma\leq\varepsilon_*(\omega)
    $
    and
    $
    \varepsilon(T_\sigma)=\sigma.
    $
    By Lemma \ref{LMsmallnessks}, $\vn{k_s}_2^2$ is decreasing on $[0,T_\sigma]$ and thus, by Proposition \ref{prop:length_sharp_estimate},
    \begin{align*}
        \varepsilon(t) \leq \varepsilon(0)\frac{L(t)^3}{L(0)^3} \leq \varepsilon(0)\left(1 + L(0)^{-4}\big( \sigma \hat{C}(\sigma,\omega) + 32\omega^4\pi^4 \big) t\right)^\frac{3}{4}.
    \end{align*}
    Letting $t=T_\sigma$ and using $T_\sigma< L(0)^4\delta_*(\varepsilon(0),\sigma)$ imply that
    \[\sigma < \varepsilon
    \Big(
    1
    + 
      \big(
      \sigma\hat{C}(\sigma, \omega) 
      + 32\omega^4\pi^4\big)\delta_*(\varepsilon(0),\sigma)
    \Big)^{\frac34}\,,
    \]
    which contradicts the definition of $\delta_*(\varepsilon(0),\sigma)$.
\end{proof}

Finally we come to the most important estimate in this section.

\begin{prop}
\label{PRPkeyestimate}
There exist positive constants $\varepsilon_1,c_2>0$ depending only $\omega$ and a universal constant $c_1>0$ with the following property:
Let $\gamma:\S^1\times[0,\infty)\rightarrow\R^2$ be a free elastic flow with $\varepsilon(0) < \varepsilon_1$.
Let $T_{\varepsilon_1}:=\inf\{t\geq0 \mid \varepsilon(t) > \varepsilon_1\}\in(0,\infty]$.
Define
\[
\SQ(t) = \frac{\ds\int k_s^2\,ds}{\ds\bigg( \int k^2\,ds \bigg)^3}
\,.
\]
Then for $t\in[0,T_{\varepsilon_1})$,
\[
\SQ(t) \le \SQ(0)\left(1+\frac{c_2}{L(0)^4}\,t\right)^{-c_1}
        \,.
\]
\end{prop}

\begin{proof}
We calculate, by using the gradient flow structure $\frac{d}{dt}E=-\|F\|_2^2$,
\begin{align*}
\SQ'(t)
    &= \frac{\ds \frac{d}{dt}\int k_s^2\,ds}{\ds\bigg( \int k^2\,ds \bigg)^3}
    + \frac{6\ds \int k_s^2\,ds\int F^2\,ds}{\ds\bigg( \int k^2\,ds \bigg)^4}
    \,.
\end{align*}
Suppose $\varepsilon_1 \leq \varepsilon_*(\omega)$ for the moment (we will fix $\varepsilon_1$ later), so that we may apply Lemma \ref{LMsmallnessks} to find for $t\in[0,T_{\varepsilon_1})$,
\begin{align*}
||k||_2^8\SQ'
\le
||k||_2^2\bigg(
-\frac1{16}\int k_{sss}^2\,ds
	 - \frac{13}{6}\kav^4 \int k_s^2\,ds
  \bigg)+ 6||k_s||_2^2||F||_2^2
\,.
\end{align*}
Integration by parts implies
\[
||F||_2^2
 = \int k_{ss}^2\,ds
  + \int k_{ss}k^3\,ds
  + \frac14\int k^6\,ds
\le \int k_{ss}^2\,ds
  + \frac14\int k^6\,ds
  \,.
\]
Substituting this in yields the estimate
\begin{align}
\notag
||k||_2^8\SQ'
&\le
||k||_2^2\bigg(
-\frac1{16}\int k_{sss}^2\,ds
	 - \frac{13}{6}\kav^4 \int k_s^2\,ds
  \bigg)
\\&\qquad 
  + 6||k_s||_2^2\int k_{ss}^2\,ds
  + \frac32||k_s||_2^2\int k^6\,ds
\,.
\label{EQkeyest1}
\end{align}
We will now systematically use \eqref{eq:Poincare} to estimate the terms with a positive coefficient and substitute the result back into \eqref{EQkeyest1}.

For the first term,
\begin{equation}
\label{EQkeyesttosub1}
6||k_s||_2^2\int k_{ss}^2\,ds
\le \frac{3\varepsilon(t)}{2\pi^2L}\int k_{sss}^2\,ds
\,.
\end{equation}

Now for the second term we prepare for the estimate by performing the average-free/average decomposition partially, on four powers of the integrand, to obtain
\begin{align*}
\frac32\int k^6\,ds
&= \frac32
    \int k^2\Big(
        (k-\kav)^4
        + 4(k-\kav)^3\kav
        + 6(k-\kav)^2\kav^2
        + 4(k-\kav)\kav^3
        + \kav^4
        \Big)\,ds
\\
&= 
    \frac32\int k^2
        (k-\kav)^4
        \,ds
    + 6\kav\int k^2
        (k-\kav)^3
        \,ds
    + 9\kav^2\int k^2
        (k-\kav)^2
        \,ds
\\&\qquad
    + 6\kav^3\int k^2
        (k-\kav)
        \,ds
+ \frac32\kav^4\int k^2\,ds
\,.
\end{align*}
Now, to estimate each of these terms we use \eqref{eq:Poincare} as in the proof of Lemma \ref{LMsmallnessks}:
\begin{equation}
\label{EQkeyesttosub2}
\frac32
    \int k^2
        (k-\kav)^4
        \,ds
\le \frac32||k-\kav||_\infty^4\,||k||_2^2
\leq \frac{3}{8\pi^2L^4}\varepsilon^2(t)\,||k||_2^2
\,,
\end{equation}
\begin{equation}
\label{EQkeyesttosub3}
6\kav\int k^2
        (k-\kav)^3
        \,ds
\le 6\frac{2\omega\pi}{L}||k-\kav||_\infty^3\,||k||_2^2
= \frac{6\omega}{\sqrt{2\pi}L^4} \varepsilon^{\frac32}(t)\,||k||_2^2
\,,
\end{equation}
\begin{equation}
\label{EQkeyesttosub4}
9\kav^2\int k^2
        (k-\kav)^2
        \,ds
\le 9\left(\frac{2\omega\pi}{L}\right)^2||k-\kav||_\infty^2\,||k||_2^2
\leq \frac{18\omega^2\pi}{L^4}\varepsilon(t)\,||k||_2^2
\,,
\end{equation}
\begin{equation}
\label{EQkeyesttosub5}
6\kav^3\int k^2
        (k-\kav)
        \,ds
\le 6\left(\frac{2\omega\pi}{L}\right)^3||k-\kav||_\infty\,||k||_2^2
\leq \frac{48\omega^3\pi^3}{\sqrt{2\pi}L^4}\varepsilon^{\frac12}(t)\,||k||_2^2
\,.
\end{equation}
For the last term, we can not estimate it; this must compete with our good term as-is.

Substituting \eqref{EQkeyesttosub1}--\eqref{EQkeyesttosub5} into \eqref{EQkeyest1} gives
\begin{align}
\notag
||k||_2^8\SQ'
&\le
\bigg(
 \frac{3\varepsilon(t)}{2\pi^2L}-\frac{||k||_2^2}{16}\bigg)\int k_{sss}^2\,ds
\\&\qquad 	
\notag
+
  \bigg( L^{-4}Q(\varepsilon(t))
  + \Big(\frac32 - \frac{13}{6}\Big)\kav^4\bigg)
  ||k||_2^2\int k_s^2\,ds
\\&=
\notag
L^{-1}\bigg(
 \frac{3\varepsilon(t)}{2\pi^2}-\frac{LE}{8}\bigg)\int k_{sss}^2\,ds
\notag
\\&\qquad 	
+
  L^{-4}\bigg( Q(\varepsilon(t))
  - \frac23(2\omega\pi)^4\bigg)
  ||k||_2^2\int k_s^2\,ds
\,.\notag
\end{align}
where
\[
Q(x)
    = \frac{3}{8\pi^2}x^2
    + \frac{6\omega}{\sqrt{2\pi}}x^{\frac32}
    + 18\omega^2\pi\,x
    + \frac{48\omega^3\pi^3}{\sqrt{2\pi}}x^{\frac12}\,.
\]

For the first term, since $LE \ge 2\pi^2$, we need
$\varepsilon(t) \le \frac{1}{6}\pi^4$ for the coefficient of the first term to be non-positive.
For the second, it only depends on the polynomial $Q$ (with $\omega$) and $\varepsilon(t)$.

We now fix $\varepsilon_1$ to be the minimum of $\varepsilon_*(\omega)$, $\frac{1}{6}\pi^4$, and the first positive root of $Q(x) - \frac1{15}(2\omega\pi)^4$.
Then
\[
||k||_2^8\SQ' \le -\frac35(16\omega^4\pi^4) L^{-4}||k||_2^2||k_s||_2^2\,,
\]
which implies
\[
\SQ(t) \le \SQ(0)\exp\bigg(
            -\frac35(16\omega^4\pi^4) \int_0^t L(\hat t)^{-4}\,d\hat t\bigg)
            \,.
\]
So now use the optimal upper bound of length in Proposition \ref{prop:length_sharp_estimate} to find
\begin{align*}
-\frac35(16\omega^4\pi^4) L^{-4}(t) 
&\le -\frac35(16\omega^4\pi^4) 
    \frac{1}{L(0)^4 + (\varepsilon_1\hat{C}(\varepsilon_1,\omega) + 32\omega^4\pi^4)\,t}
\\
&= -\frac{c_1}{c_2^{-1}L(0)^4+t},
\end{align*}
where we have set $c_1 = 48\omega^4\pi^4/(5c_2)$ and $c_2 = \varepsilon_1\hat C + 32\omega^4\pi^4$; in fact, since $\varepsilon_1\hat{C}(\varepsilon_1,\omega)$ is bounded by a quadratic polynomial in $\omega$, we may even replace $c_1$ with a universal positive constant.
Therefore,
\begin{align*}
    \SQ(t) &\le \SQ(0)\exp\Big(
        -c_1\log \big( c_2^{-1}L(0)^4+t \big)
        +c_1\log \big( c_2^{-1}L(0)^4 \big)\Big)\,,
\end{align*}
which is equal to the assertion after rearrangement.
\end{proof}


Then we will have preservation of \eqref{EQgeomsmall}, and in fact
$\varepsilon(t) \rightarrow 0$.

\begin{cor}
\label{CYpres}
There exists $\varepsilon_2\in(0,\varepsilon_1)$ depending only on $\omega$ with the following property:
Let $\gamma:\S^1\times[0,\infty)\rightarrow\R^2$ be a free elastic flow with $\varepsilon(0) \leq \varepsilon_2$.
Then for all $t\in[0,\infty)$ we have $\varepsilon(t)\le\varepsilon_1$, and moreover
\[
\varepsilon(t) \le
    c_3\left(1+\frac{t}{L(0)^4}\right)^{-c_1},
\]
where $c_3>0$ depends only on $\omega$.
In particular, $\varepsilon(t)\rightarrow0$ as $t\rightarrow\infty$.
\end{cor}

\begin{proof}
Using Proposition \ref{PRPkeyestimate} and the average-free/average decomposition, we estimate for $t\subset[0,T_{\varepsilon_1})$,
\begin{align*}
\varepsilon(t)
= L^3||k||_2^6\SQ(t) &\le \left(\frac{\varepsilon(t)}{8\pi^2} + 2\omega^2\pi^2\right)^3 \SQ(0) \left(1+\frac{c_2}{L(0)^4}\,t\right)^{-c_1}\\
&\le \left(\frac{\varepsilon_1}{8\pi^2} + 2\omega^2\pi^2\right)^3 \SQ(0) \left(1+\frac{c_2}{L(0)^4}\,t\right)^{-c_1}\,.
\end{align*}
Furthermore, noting that
$
\SQ(0)=\frac{\varepsilon(0)}{L(0)^3(2E(0))^3}\leq \frac{\varepsilon(0)}{64\pi^6},
$
we obtain
\begin{align*}
\varepsilon(t) \le \frac{\varepsilon(0)}{64\pi^6} \left( \frac{\varepsilon_1}{8\pi^2} + 2\omega^2\pi^2 \right)^3 \left(1+\frac{c_2}{L(0)^4}\,t\right)^{-c_1}
\,.
\end{align*}
Now, since $\delta_*(\varepsilon,\varepsilon_1)\nearrow\infty$ as $\varepsilon\searrow0$,
we can take small $\varepsilon_2\in(0,\varepsilon_1)$ depending only on $\omega$ such that
\[
\frac{1}{64\pi^6}\left(\frac{\varepsilon_1}{8\pi^2} + 2\omega^2\pi^2\right)^3\left(1+\frac{\delta_*(\varepsilon_2,\varepsilon_1)}{c_2}\right)^{-c_1}\leq 1\,,
\]
Since $L(0)^4\delta_*(\varepsilon_2,\varepsilon_1) \leq L(0)^4\delta_*(\varepsilon(0),\varepsilon_1) \leq T_{\varepsilon_1}$ by Lemma \ref{lem:delta_*}, we have
\[
\varepsilon(0)\leq\varepsilon_2 \quad \Longrightarrow \quad 
\begin{cases}
    \varepsilon\big (L(0)^4\delta_*(\varepsilon_2,\varepsilon_1) \big) \leq \varepsilon_2,\\
    \text{$\varepsilon(t)\leq \varepsilon_1$ for all $t\in[0,L(0)^4\delta_*(\varepsilon_2,\varepsilon_1)]$.}
\end{cases}
\]
This together with the uniform length lower bound
$L(t)\geq 2\pi^2E(0)^{-1}$
ensures that we can now recursively use this property to obtain for all $t\in[0,\infty)$,
\begin{align*}
\varepsilon(t) 
\le \varepsilon_1.
\end{align*}
Hence $T_{\varepsilon_1}=\infty$, so taking suitable $c_3$ completes the proof.
\end{proof}

Therefore the scale-invariant curvature must approach a constant.
An appropriate rescaling of the flow must approach a circle.
In fact, Corollary \ref{CYpres} shows already that $||k^\eta_s||_2^2(t)=\varepsilon(t)\rightarrow0$ for the rescaled flow $\eta$.

\begin{rmk}
    We also record the following sharp scale-invariant estimate:
    \[
    L(t)E(t) \le \frac{\varepsilon(t)}{8\pi^2} + 2\omega^2\pi^2
    \,.
    \]
    The proof follows easily by using the definition of $\varepsilon(t)$ and the average-free/average decomposition.
    This estimate combined with Proposition \ref{prop:length_sharp_estimate} implies that at least if $\varepsilon(t)$ remains small, then the elastic energy decays like $E(t)\lesssim t^{-\frac{1}{4}}$.
    The exponent $-\frac{1}{4}$ is optimal in view of \eqref{eq:length_sublinear}.
\end{rmk}

We continue our estimate to control the centre of the final circle.
From now on the lower bound part of the sharp length estimate in Proposition \ref{prop:length_sharp_estimate} comes into play.
For this reason, taking smaller $\varepsilon_1$ if necessary, we assume hereafter that
\[
\varepsilon_1\hat{C}(\varepsilon_1,\omega)\leq 16\omega^4\pi^4,
\]
so that the length lower bound also has a positive coefficient of $t$.

\begin{prop}
\label{PRestf}
Let $\gamma:\S^1\times[0,\infty)\rightarrow\R^2$ be a free elastic flow with $\varepsilon(0) < \varepsilon_2$.
Then for some $C>0$ depending on $\omega$ and $L(0)$,
\[
\|L(t)k(\cdot,t)-2\omega\pi\|_\infty \leq C\left(1+t\right)^{-c_1/2},
\]
and furthermore
\[
    ||k(\cdot,t)^3||_\infty \le 
    C (1+t )^{-(3+2c_1)/4} + \frac{
    8\omega^3\pi^3}{\Big( L(0)^4 + \big( 32\omega^4\pi^4-\varepsilon_1\hat{C}(\varepsilon_1,\omega) \big)\,t \Big)^{3/4} }\,.
\]
\end{prop}

\begin{proof}
The first assertion follows from Corollary \ref{CYpres} with the simple estimate
\[
\vn{Lk-2\omega\pi}_\infty
= L\vn{k - \kav}_\infty
\le \sqrt{\frac{\varepsilon(t)}{2\pi}}.
\]
Similarly, using also Corollary \ref{CYpres},
\begin{align*}
\vn{L(t)^3k(\cdot,t)^3}_\infty
&=
    \vn{ (Lk-L\kav)^3
    + 3L\kav(Lk-L\kav)^2
    + 3L^2\kav^2(Lk-L\kav)
    + L^3\kav^3 }_\infty
\\
&\le
    (\varepsilon(t)/2\pi)^{\frac32}
    + 3(2\omega\pi)(\varepsilon(t)/2\pi)
    + 3(2\omega\pi)^2(\varepsilon(t)/2\pi)^{\frac12}
    + (2\omega\pi)^3\,
\\
&\le
    \sqrt{\varepsilon(t)}(\varepsilon_1/2\pi
    + 3(2\omega\pi)(\sqrt{\varepsilon_1}/2\pi)
    + 3(2\omega\pi)^2(1/2\pi)^{\frac12})
    + (2\omega\pi)^3
\\
&\le
    C(1+t)^{-c_1/2}
    + 8\omega^3\pi^3
\,.
\end{align*}
The lower length bound in Proposition \ref{prop:length_sharp_estimate} (for $\sigma=\varepsilon_1$) with Corollary \ref{CYpres} gives
\[
    L(t)^{-3} \le (L(0)^4 + (32\omega^4\pi^4-\varepsilon_1\hat{C}(\varepsilon_1,\omega))\,t)^{-3/4}
    \,.
\]
Combining the previous two estimates gives the result.
\end{proof}

\begin{lem}
\label{LMkssest}
Let $\gamma:\S^1\times[0,\infty)\rightarrow\R^2$ be a free elastic flow with $\varepsilon(0) < \varepsilon_2$.
Then for any integer $m\geq1$,
\[
    ||k_{s^m}||_\infty
    \le C(1+t)^{-(m+1+c_1)/4}
\,,
\]
where $C$ depends on $m$, $\omega$, $L(0)$, and $||k_{s^{2m+1}}||_2^2(0)$.
\end{lem}

\begin{proof}
    We apply Dziuk--Kuwert--Sch\"atzle's interpolation arguments (as in \cite[(3.6), (2.16)]{DKS02}) to deduce for all $\ell\geq1$,
    \[
    \frac{d}{dt}\int k_{s^\ell}^2\,ds
     + \int k_{s^{\ell+2}}^2\,ds
     \le C_\ell||k||_2^{4\ell+10}
    \,.
    \]
    Using the curvature estimate $\vn{k}_\infty \leq CL^{-1}$ and then the length estimate, we have
    \[
    \frac{d}{dt}\int k_{s^\ell}^2\,ds
     \le CL^{-(2\ell+5)}
     \le \frac{C}{(L(0)^4 + (32\omega^4\pi^4-\varepsilon_1\hat{C}(\varepsilon_1,\omega))\,t)^{(2\ell+5)/4}}
        \,.
    \]
    (Here and hereafter $C$ will change line by line.)
    This implies
    \[
    \int k_{s^\ell}^2\,ds
    \le
    \frac{C}{(1+t)^{(2\ell+1)/4}}
    \,,
    \]
    where $C$ depends additionally on $L(0)$ and $||k_{s^\ell}||_2^2(0)$.
    
    Now, combining the above estimates (for $\ell=2m+1$) with our previous estimates, namely \eqref{eq:Poincare}, Proposition \ref{prop:length_sharp_estimate} (implying $L^4\geq C(1+t)$), and Corollary \ref{CYpres}, we find
    \begin{align*}
    ||k_{s^m}||_\infty^2
    &\le \frac{L}{2\pi}||k_{s^{m+1}}||_2^2
    \\
    &\le \frac{L}{2\pi}||k_{s}||_2||k_{s^{2m+1}}||_2
    \\
    &\le CL^{-1/2}\varepsilon(t)^{1/2}(1+t)^{-(4m+3)/8}
    \\
    &\le C(1+t)^{(-1/2)(1/4)+(1/2)(-c_1)-(4m+3)/8} \\
    &= C(1+t)^{-(m+1+c_1)/2}
    \end{align*}
    as required.
\end{proof}

Now we conclude with our desired convergence result.

\begin{proof}[Proof of Theorem \ref{thm:stability}]
Take $\varepsilon:=\varepsilon_2$ as in Corollary \ref{CYpres}.
The curvature convergence from Proposition \ref{PRestf} implies that the image $\eta(\S^1,t)$ is becoming closer to an $\omega$-circle with radius $\frac{1}{2\omega\pi}$.
In addition, the convergence is smooth since Lemma \ref{LMkssest} with Proposition \ref{prop:length_sharp_estimate} implies the decay of higher order derivatives of rescaled curvature:
\[
\|k^\eta_{s^m}\|_\infty = L^{m+1}\|k_{s^m}\|_\infty \leq C(1+t)^{-c_1/4} \to0
\]
as $t\to\infty$.
Hence, the continuous translation
\[
\eta(\cdot,t) - \int_{\S^1} \eta(\cdot,t)\,ds
\]
clearly has the claimed convergence property.

In order to do without the continuous translation, it is sufficient to prove that the limit image is confined in the desired disk; that is,
\[
\limsup_{t\to\infty} \vn{\eta(\cdot,t)}_\infty \leq \frac{1}{2\omega\pi}.
\]
We prove this by giving more estimates on the rescaled flow.
From the evolution equation, Proposition \ref{PRestf}, and Lemma \ref{LMkssest} (for $m=1$), we have
\begin{align*}
|\gamma|(\cdot,t)
&\le |\gamma|(\cdot,0)
    + \int_0^t |F|(\cdot,\hat t)\,dt
\\
&\le \vn{\gamma}_\infty(0)
    + \int_0^t \vn{k_{ss}}_\infty(\hat{t}) \,d\hat{t} + \frac{1}{2}\int_0^t \vn{k^3}_\infty(\hat{t}) \,d\hat{t}
\\
&\le \vn{\gamma}_\infty(0)
    + C(1+t)^{(1-c_1)/4} + C(1+t)^{(1-2c_1)/4}
\\&\qquad
    + \frac{16\omega^3\pi^3}{32\omega^4\pi^4-\varepsilon_1\hat{C}(\varepsilon_1,\omega)}
    \Big( \big( L(0)^4 + ( 32\omega^4\pi^4-\varepsilon_1\hat{C}(\varepsilon_1,\omega) )\,t \big)^{1/4}
     - L(0)\Big)
\,.
\end{align*}
Multiplying through by $1/L$ and using Proposition \ref{prop:length_sharp_estimate}, we find
\begin{align*}
|\eta|(\cdot,t)
&\le \vn{\gamma}_\infty(0)L(t)^{-1}
    + C(1+t)^{-c_1/4} + C(1+t)^{-2c_1/4}
\\&\qquad
    + \frac{16\omega^3\pi^3}{32\omega^4\pi^4-\varepsilon_1\hat{C}(\varepsilon_1,\omega)}
    \frac{\big( L(0)^4 + (32\omega^4\pi^4-\varepsilon_1\hat{C}(\varepsilon_1,\omega))\,t \big)^{1/4}
     - L(0)}
     {(L(0)^4 + \big( 32\omega^4\pi^4-\varepsilon_1\hat{C}(\varepsilon_1,\omega))\,t \big)^{1/4}}
\,.
\end{align*}
Therefore,
\[
|\eta|(\cdot,t)
\le \frac{16\omega^3\pi^3}{32\omega^4\pi^4-\varepsilon_1\hat{C}(\varepsilon_1,\omega)}
+ f(t)\,,
\]
where $f(t)\rightarrow0$ as $t\to\infty$.
This already gives an effective boundedness for $\eta$.

Now we use the decay estimate for $\varepsilon(t)$ in Corollary \ref{CYpres} to redefine $\varepsilon_1$.
More precisely, applying the above argument to the new initial time $t=M$, and also new $\varepsilon_1=\varepsilon_1(M)$ such that $\varepsilon_1(M)\to0$ as $M\to\infty$, we further have
\[
|\eta|(\cdot,t+M)
\le \frac{16\omega^3\pi^3}{32\omega^4\pi^4-\varepsilon_1(M)\hat{C}(\varepsilon_1(M),\omega)}
+ f_M(t)\,,
\]
where $f_M$ still has the property that $f_M(t)\to0$ as $t\to\infty$.
Taking $t\to\infty$ gives
\[
\limsup_{t\to\infty}\vn{\eta(\cdot,t)}_\infty=\limsup_{t\to\infty}\vn{\eta(\cdot,t+M)}_\infty
\le \frac{16\omega^3\pi^3}{32\omega^4\pi^4-\varepsilon_1(M)\hat{C}(\varepsilon_1(M),\omega)}\,,
\]
and then taking $M\to\infty$ implies
\[
\limsup_{t\to\infty}\vn{\eta(\cdot,t)}_\infty
\le \frac{16\omega^3\pi^3}{32\omega^4\pi^4} = \frac{1}{2\omega\pi}\,.
\]
The proof is now complete.
\end{proof}

\bibliographystyle{plain}
\bibliography{FEF}

\end{document}